\documentclass[12pt, amscd]{amsart}
\usepackage{amscd}
\usepackage{verbatim}

\input xy
\xyoption{all}

\usepackage{amssymb}
\usepackage{mathrsfs}

\DeclareMathAlphabet{\cat}{OT1}{cmss}{m}{sl}

\newtheorem{theorem}{Theorem}[section]

\newtheorem{lemma}[theorem]{Lemma}
\newtheorem{corollary}[theorem]{Corollary}

\theoremstyle{definition}

\newcommand{\xra}{\xrightarrow}

\newcommand{\tens}{\otimes}

\newcommand{\Gr}{\operatorname{Gr}}
\newcommand{\Fl}{\operatorname{Fl}}

\newcommand{\Hom}{\operatorname{Hom}}

\newcommand{\D}{\operatorname{D}}

\newcommand{\rank}{\operatorname{rank}}

\newcommand{\cT}{\mathcal T}

\newcommand{\cS}{\mathcal S}

\usepackage[hypertex]{hyperref}

\title[Semiorthogonal decompositions for twisted grassmannians] 
{Semiorthogonal decompositions for twisted grassmannians}

\author
[S. Baek] {Sanghoon Baek}

\email {baek.sanghoon@gmail.com}

\begin{document}
\begin{abstract}
In this article, we present semiorthogonal decompositions for twisted forms of grassmannians.

\end{abstract}

\maketitle

\section{introduction}

In \cite{Orl} Orlov gave the semiorthogonal decompositions for projective, grassmann, and flag bundles, which generalize the full exceptional collections on the corresponding varieties by Beilinson \cite{Bei} and Kapranov \cite{Kap}. 

In the case of projective bundles, Bernardara \cite{Ber} extended the semiorthogonal decomposition to the twisted forms. In this paper, we present, in a similar way,  semiorthogonal decompositions for twisted forms of grassmannians. The proof of the main theorem uses the study of $K$-theory of twisted grassmannians by Levine-Srinivas-Weyman \cite{LSW} and Panin \cite{Pan}.

\bigskip

{\bf Notation and Conventions} In this paper, a partition means a nonincreasing sequence of numbers. Given a partition $\alpha=(\alpha_{i})$, the conjugate of $\alpha$, written by $\alpha^{*}$, is the partition $(\alpha^{*}_{j})$, where $\alpha^{*}_{j}=|\{i\,|\, \alpha_{i}\geq j\}|$. For a scheme $X$ over $Y$, we denote by $\Delta(X/Y)$ the diagonal of $X\!\times_{Y}\!X$. We abbreviate $\mathscr{F}\boxtimes \mathscr{G}=\pi_{1}^{*}\mathscr{F}\tens \pi_{2}^{*}\mathscr{G}$ for sheaves $\mathscr{F},\mathscr{G}$ over $X$, where $\pi_{i}$ is the projection of $X\!\times_{Y}\!X$. For a scheme $X$, we denoted by $\D(X)$ the bounded derived category of coherent sheaves on $X$.

\section{Preliminaries}

In this section we recall some definitions from \cite{Bon}. Let $F$ be a field. We fix a triangulated $F$-linear category $\cT$.

Given a full subcategory $\cS$ of $\cT$, the (right) \emph{orthogonal complement}, denoted by $\cS^{\bot}$, is a full subcategory whose objects are objects $T$ in $\cT$ satisfying
\[\Hom_{\cT}(S,T)=0\]
for all $S\in \cS$.

A full triangulated subcategory $\cS$ of $\cT$ is (right) \emph{admissible} if for any $T\in \cT$ there is a distinguished triangle
\[S\to T\to S^{\bot}\to S[1]\]
with $S\in \cS$ and $S^{\bot}\in \cS^{\bot}$, i.e. $\cT=\langle \cS, \cS^{\bot}\rangle$.

Now, we consider a sequence $(\cS_{1}, \ldots, \cS_{n})$ of admissible subcategories of $\cT$. This sequence is called \emph{semiorthogonal} if
\begin{equation}\label{semiorthogonal}
\cS_{i}\subset \cS_{j}^{\bot}
\end{equation}
for $1\leq i<j\leq n$. A semiorthogonal sequence $(\cS_{1}, \ldots, \cS_{n})$ of $\cT$ is called \emph{semiorthogonal decomposition} if $\cT=\langle \cS_{1}, \ldots, \cS_{n}\rangle$.

For example, the sequence $(\cS, \cS^{\bot})$ for an admissible full triangulated subcategory $\cS$ gives a semiorthogonal decomposition of $\cT$. By definition, any generating sequence $(\cS_{1}, \ldots, \cS_{n})$ of full subcategories of $\cT$ satisfying $(\ref{semiorthogonal})$ is a semiorthogonal decomposition of $\cT$.

\section{Twisted Grassmannians}

Let $X$ be a Noetherian scheme of characteristic $0$ and let $\mathscr{A}$ be a sheaf of Azumaya algebras of rank $n^{2}$ over $X$. For an integer $1\leq k< n$, a twisted grassmannian $p:\Gr(k,\mathscr{A})\to X$ is defined by the representable functor from the category $\cat{Schemes/X}$ of schemes over $X$ to the category of $\cat{Sets}$ of sets given by $(Y\stackrel{\phi}\to X)\mapsto$ the set of sheaves of left ideals $\mathscr{I}$ of $\phi^{*}\mathscr{A}$ such that $\phi^{*}\mathscr{A}/\mathscr{I}$ is a locally free $\mathscr{O}_{Y}$-modules of rank $n(n-k)$.

There is an \'etale covering $i:U\to X$ and a locally free sheaf $\mathscr{E}$ of rank $n$ over $U$ with the following pullback diagram
$$\xymatrixcolsep{4pc}
\xymatrix{
\Gr(k,\mathscr{E}\!nd(\mathscr{E})) \ar[r]^{j} \ar[d]^{q} & \Gr(k,\mathscr{A})\ar[d]^{p}  \\
U \ar[r]^{i} & X,
}
$$
where $\Gr(k,\mathscr{E}\!nd(\mathscr{E}))$ is naturally isomorphic to $\Gr(k,\mathscr{E})$.

Consider the tautological exact sequence of sheaves on $\Gr(k,\mathscr{E})$
\[0\to\mathscr{R}\to q^{*}\mathscr{E}\to \mathscr{T}\to 0,\]
where $\rank(\mathscr{R})=k$. For a partition $\alpha=(\alpha_1,\cdots, \alpha_{k})$ with $0\leq \alpha_i\leq n-k$, we denoted by $S^{\alpha}$ the Schur functor for $\alpha$. We define $S(\alpha)$ to be the full subcategory of $\D(\Gr(k,\mathscr{A}))$ generated by $\mathscr{M}$ in $\D(\Gr(k,\mathscr{A}))$ satisfying
\[\mathscr{M}\!\!\mid_{\Gr(k,\mathscr{E})}\simeq q^{*}\!\mathscr{N}\tens S^{\alpha}\!\mathscr{R},\]
for some $\mathscr{N}\in \D(U)$.

\begin{lemma}\label{orthogonality}
Let $\alpha, \alpha'$ be two distinct partitions with $0\leq\alpha_i,\alpha'_{i}\leq n-k$ and let $\mathscr{M}\in S(\alpha)$ and $\mathscr{M}'\in S(\alpha')$. Then $R\mathscr{H}\!om(\mathscr{M},\mathscr{M}')=0$.
\end{lemma}
\begin{proof}
Let $\mathscr{M}\!\!\mid_{\Gr(k,\mathscr{E})}\simeq q^{*}\!\mathscr{N}\tens S^{\alpha}\!\mathscr{R}$ and $\mathscr{M'}\!\!\mid_{\Gr(k,\mathscr{E})}\simeq q^{*}\!\mathscr{N'}\tens S^{\alpha'}\!\mathscr{R}$.

By the Littlewood-Richardson rule, the partition $\beta$ of an irreducible summand $S^{\beta}\mathscr{R}$ of $\mathscr{H}\!om(S^{\alpha}\mathscr{R},S^{\alpha'}\mathscr{R})$ is of the form $(\beta_{1},\cdots, \beta_{k})$ with $-(n-k)\leq \beta_{i}\leq n-k$. Hence, by the Borel-Bott-Weil theorem as in \cite[Lemma 3.2]{Kap} we have
\begin{equation}\label{trivial}
Rq_{*}(\mathscr{H}\!om(S^{\alpha}\!\mathscr{R},S^{\alpha'}\!\mathscr{R}))=0.
\end{equation}

It is enough to show the result locally. By the adjoint property of $Rq_{*}$ and $q^{*}$, projection formula, and (\ref{trivial}), we have
\begin{align*}
R\mathscr{H}\!om(\mathscr{M}\!\!\mid_{\Gr(k,\mathscr{E})},\mathscr{M}'\!\!\mid_{\Gr(k,\mathscr{E})})&=R\mathscr{H}\!om(q^{*}\!\mathscr{N},q^{*}\!\mathscr{N'}\tens \mathscr{H}\!om(S^{\alpha}\!\mathscr{R},S^{\alpha'}\!\mathscr{R}))\\
&=R\mathscr{H}\!om(\mathscr{N},Rq_{*}(q^{*}\!\mathscr{N'}\tens \mathscr{H}\!om(S^{\alpha}\!\mathscr{R},S^{\alpha'}\!\mathscr{R})))\\
&=R\mathscr{H}\!om(\mathscr{N},\mathscr{N'}\tens Rq_{*}(\mathscr{H}\!om(S^{\alpha}\!\mathscr{R},S^{\alpha'}\!\mathscr{R})))\\
&=0.\qedhere
\end{align*}
\end{proof}

\begin{theorem}\label{grassthm}
Let $(S(\alpha)\,|\, \alpha=(\alpha_{1},\cdots,\alpha_{k}), 0\leq \alpha_{i}\leq n-k)$ be a sequence of the full subcategories of $\D(\Gr(k,\mathscr{A}))$ by the lexicographical order on $\alpha$. Then this sequence gives a semiorthogonal decomposition of $\D(\Gr(k,\mathscr{A}))$.
\end{theorem}
\begin{proof}
By Lemma \ref{orthogonality}, it suffices to show that $(S(\alpha))$ generates $\D(\Gr(k,\mathscr{A}))$. Following \cite{LSW} or \cite{Pan}, there exist sheaves $\mathscr{F}_{\alpha}$ of right $\mathscr{A}^{|\alpha|}$-modules and sheaves $\mathscr{G}_{\alpha^{*}}$ of left $\mathscr{A}^{|\alpha|}$-modules such that
\[j^{*}\!\mathscr{F}_{\alpha}\simeq S^{\alpha}\!\mathscr{R}\tens q^{*}((\mathscr{E}^{*})^{\tens |\alpha|}),\,j^{*}\!\mathscr{G}_{\alpha^{*}}\simeq q^{*}(\mathscr{E}^{\tens |\alpha|})\tens S^{\alpha^{*}}\!\mathscr{T}^{*}.
\]
Moreover, the sequence $\mathscr{R} \boxtimes\! \mathscr{T}^{*}\to \mathcal{O}_{\Gr(k,\mathscr{E})\times \Gr(k,\mathscr{E})}\to \mathcal{O}_{\Delta(\Gr(k,\mathscr{E})/U)}\to 0$ descends to the sequence $\mathscr{F}_{(1)} \boxtimes\, \mathscr{G}_{(1)}\to \mathcal{O}_{\Gr(k,\mathscr{A})\times \Gr(k,\mathscr{A})}\to \mathcal{O}_{\Delta(\Gr(k,\mathscr{A})/X)}\to 0$. Hence, we have the Koszul resolution:
\[0\to \Lambda^{k(n-k)}(\mathscr{F}_{(1)} \boxtimes\, \mathscr{G}_{(1)})\to \Lambda^{k(n-k)-1}(\mathscr{F}_{(1)} \boxtimes\, \mathscr{G}_{(1)})\to \cdots\]
\[\cdots \to \mathscr{F}_{(1)} \boxtimes\, \mathscr{G}_{(1)}\to \mathcal{O}_{\Gr(k,\mathscr{A})\times_{X} \Gr(k,\mathscr{A})}\to \mathcal{O}_{\Delta(\Gr(k,\mathscr{A})/X)}\to 0.\]
Therefore, as $\Lambda^{m}(\mathscr{F}_{(1)} \boxtimes\, \mathscr{G}_{(1)})=\bigoplus_{|\alpha|=m} \mathscr{F}_{\alpha} \boxtimes\, \mathscr{G}_{\alpha^{*}}$ for $1\leq m\leq k(n-k)$, the sheaf of ideals of the diagonal embedding $\mathcal{O}_{\Delta(\Gr(k,\mathscr{A})/X)}$ is in the subcategory \[\langle \pi_{1}^{*}\mathscr{F}_{\alpha}\tens \pi_{2}^{*}\mathscr{G}_{\alpha^{*}}\,|\, 0\leq |\alpha|\leq k(n-k)\rangle\] of $D(\Gr(k,\mathscr{A})\times \Gr(k,\mathscr{A}))$. By the projection formula, for any $\mathscr{M}\in D(\Gr(k,\mathscr{A}))$ we have $\mathscr{M}=R(\pi_{1})_{*}(\pi_{2}^{*}\mathscr{M}\tens \mathcal{O}_{\Delta(\Gr(k,\mathscr{A})/X)})$.

To finish the proof, it is enough to verify that \[R(\pi_{1})_{*}\big(\pi_{2}^{*}\mathscr{M}\tens (\pi_{1}^{*}\mathscr{F}_{\alpha}\tens \pi_{1}^{*}\mathscr{G}_{\alpha^{*}})\big)\in S(\alpha):\] since we have
$R(\pi_{1})_{*}\big(\pi_{2}^{*}\mathscr{M}\tens (\pi_{1}^{*}\mathscr{F}_{\alpha}\tens \pi_{1}^{*}\mathscr{G}_{\alpha^{*}})\big)=R(\pi_{1})_{*}\big(\pi_{2}^{*}(\mathscr{M}\tens \mathscr{G}_{\alpha^{*}})\big)\tens \mathscr{F}_{\alpha}$, this is isomorphic to
\[q^{*}\big(Rq_{*}(\mathscr{M}\tens S^{\alpha^{*}}\!\mathscr{T})\big)\tens S^{\alpha}\!\mathscr{R} \text{ over } \Gr(k,\mathscr{E}).\qedhere\]
\end{proof}

\bigskip

Let $1\leq k_{1}<\cdots <k_{m}<n$ be a sequence of integers. Given a sheaf of Azumaya algebras $\mathscr{A}$ of rank $n^{2}$ over $X$, we denote by $\Fl(k_1,\cdots, k_m, \mathscr{A})$ the functor defined by  $(Y\stackrel{\phi}\to X)\mapsto$ the set of sheaves of left ideals $\mathscr{I}_{1}\subset \cdots \subset \mathscr{I}_{m}$ of $\phi^{*}\mathscr{A}$ such that $\phi^{*}\mathscr{A}/\mathscr{I}_{i}$ is a locally free $\mathscr{O}_{Y}$-modules of rank $n(n-k_{i})$.

As in the grassmannian case, we have  an \'etale covering $i:U\to X$ and a locally free sheaf $\mathscr{E}$ of rank $n$ over $U$ with the following pullback diagram
$$\xymatrixcolsep{4pc}
\xymatrix{
\Fl(k_{1}, \cdots, k_{m},\mathscr{E}\!nd(\mathscr{E})) \ar[r]^{j} \ar[d]^{q} & \Fl(k_{1}, \cdots, k_{m},\mathscr{A})\ar[d]^{p}  \\
U \ar[r]^{i} & X,
}
$$
where $\Fl(k_{1}, \cdots, k_{m},\mathscr{E}\!nd(\mathscr{E}))$ is identified with $\Fl(k_{1}, \cdots, k_{m},\mathscr{E})$. We have the tautological flags
\[\mathscr{R}_{1}\hookrightarrow \cdots \hookrightarrow \mathscr{R}_{m}\hookrightarrow q^{*}\mathscr{E}\twoheadrightarrow \mathscr{T}_{1}\twoheadrightarrow \cdots \twoheadrightarrow\mathscr{T}_{m},\]
where $\rank(\mathscr{R}_{i})=k_{i}$ and $\rank(\mathscr{T}_{i})=n-k_{i}$.

Let $\alpha(1), \cdots, \alpha(m-1), \alpha(m)$ be partitions of the forms \[(\alpha_1, \cdots, \alpha_{k_{1}}), \cdots, (\alpha_1, \cdots, \alpha_{k_{m-1}}), (\alpha_1, \cdots, \alpha_{k_{m}})\] with $0\leq \alpha_i\leq k_{2}-k_{1}, \cdots, 0\leq \alpha_i\leq k_{m}-k_{m-1}, 0\leq \alpha_i\leq n-k_{m}$, respectively. For $1\leq i\leq m$, we define $S(\alpha(1),\cdots,\alpha(m))$ to be the full subcategory of $\D(\Fl(k_{1}, \cdots, k_{m},\mathscr{A}))$ generated by $\mathscr{M}$ in $\D(\Fl(k_{1}, \cdots, k_{m},\mathscr{A}))$ satisfying
\[\mathscr{M}\!\!\mid_{\Fl(k_{1}, \cdots, k_{m},\mathscr{E})}\simeq q^{*}\!\mathscr{N}\tens S^{\alpha(1)}\!\mathscr{R}_{1}\tens \cdots \tens S^{\alpha(m)}\!\mathscr{R}_{m},\]
for some $\mathscr{N}\in \D(U)$.

\begin{corollary}
Let $1\leq k_{1}<\cdots <k_{m}<n$ be a sequence of integers and let $(S(\alpha(1),\cdots,\alpha(m))\,|\, \text{ all partitions of the form } \alpha(i), 1\leq i\leq m )$ be a sequence of the full subcategories of $\D(\Fl(k_{1}, \cdots, k_{m},\mathscr{A}))$ in lexicographical order. Then this gives a semiorthogonal decomposition of $\D(\Fl(k_{1}, \cdots, k_{m},\mathscr{A}))$.
\end{corollary}
\begin{proof}
We shall prove by induction on $m$. The case $m=1$ follows from Theorem \ref{grassthm}. Assume that the result holds for $m-1$. There are projections
\[\Fl(k_{1}, \cdots, k_{m},\mathscr{E})\xra{q_{m}}\cdots \xra{q_{2}}\Fl(k_{m},\mathscr{E})\xra{q_{1}} U\]
and
\[\Fl(k_{1}, \cdots, k_{m},\mathscr{A})\xra{p_{m}}\cdots \xra{p_{2}}\Fl(k_{m},\mathscr{A})\xra{p_{1}} X.\]
Let $\mathscr{R}'_{2}\subset (q_{1}\circ \cdots \circ q_{m-1})^{*}\mathscr{E}$ be the tautological subsheaf over $\Fl(k_{2}, \cdots, k_{m},\mathscr{E})$ and let $\mathscr{A}'$ be the sheaf of Azumaya algebra over $\Fl(k_{2}, \cdots, k_{m},\mathscr{A})$ from $\mathscr{E}\!nd(\mathscr{R}'_{2})$ by descent. Then, we have $\Fl(k_{1}, \cdots, k_{m},\mathscr{E})=\Gr_{\Fl(k_{2}, \cdots, k_{m},\mathscr{E})}(k_{1},\mathscr{R}'_{2})$ and $\Fl(k_{1}, \cdots, k_{m},\mathscr{A})=\Gr_{\Fl(k_{2}, \cdots, k_{m},\mathscr{A})}(k_{1},\mathscr{A}')$. Now the result follows from the proofs of Lemma \ref{orthogonality} and Theorem \ref{grassthm}.
\end{proof}

\end{document}